\documentclass[review]{elsarticle}

\makeatletter
\def\ps@pprintTitle{%
	\let\@oddhead\@empty
	\let\@evenhead\@empty
	\def\@oddfoot{\centerline{\thepage}}%
	\let\@evenfoot\@oddfoot}
\makeatother

\usepackage{lineno,hyperref}
\usepackage{amssymb,enumerate}
\usepackage{amsmath}
\usepackage{graphics}
\usepackage{amsthm}
\numberwithin{equation}{section}

\def\R{{\bf R}}

\def\e{{\varepsilon}}

\newtheorem{thm}{Theorem}[section]

\newtheorem{lem}{Lemma}[section]

\newtheorem{rem}{Remark}[section]

\begin{document}
	\begin{frontmatter}
		
		\title{Global Instability of the Multi-dimensional Plane Shocks for the isothermal flow}
		
		\author[mymainaddress,secondaryaddress]{Ning-An Lai}
		\ead{hyayue@gmail.com}
		
		\author[thirdlyaddress]{Wei, Xiang}
		\ead{weixiang@cityu.edu.hk}
		
		\author[secondaryaddress]{Yi Zhou}
		\ead{yizhou@fudan.edu.cn}
		
		\address[mymainaddress]{Institute of Nonlinear Analysis and Department of Mathematics,\\ Lishui University, Lishui 323000, China}
		\address[secondaryaddress]{School of Mathematical Sciences, Fudan University, Shanghai 200433, China}
		\address[thirdlyaddress]{Department of Mathematics, City University of Hong Kong, Kowloon, Hong Kong 999077, People¡¯s Republic of China }

		\begin{abstract}
In this paper, we are concerned with the long time behavior of the piecewise smooth solutions to the generalized Riemann problem governed by the compressible isothermal Euler equations in two and three dimensions. Non-existence result is established for the fan-shaped wave structure solution, including two shocks and one contact discontinuity and which is a perturbation of plane waves. Therefore, unlike the one-dimensional case, the multi-dimensional plane shocks are not stable globally. What is more, the sharp lifespan estimate is established which is the same as the lifespan estimate for the nonlinear wave equations in both two and three space dimensions. 
		\end{abstract}

		
	\end{frontmatter}


\section{Introduction}
We are concerned with the non-existence of global solutions of the generalised Riemann problem governed by the compressible isothermal Euler equations. 
More precisely, we prove that the multi-dimensional $(N=2, 3)$ plane shocks are not stable in the global sense with respect to a smooth perturbation. 

It is well-known that smooth solutions of the compressible Euler equations with some compression assumption will generate singularity in finite time no matter how small the initial data is (cf. \cite{ChMiao,lax,LiZhouKong,Liu,Si}). Therefore, it is natural and important to study the Cauchy problem with discontinuous initial data. For the one-dimensional case, a satisfactory theory on the global existence and stability of the Cauchy problem is established by many mathematicians (cf. \cite{BiBr}-\cite{ChenChZhang},\cite{ChenXiangZhang,Da,FRX,Glimm,Lax1,Liu1,LiuYang}). However, the $B.V.$ space is not a wellposed space any more for the Cauchy problem in multidimensions. So almost all the efforts are focused on the multidimensional Riemann problem or the structural stability of important physical problems introduced in Courant-Friedrichs classic book \cite{CoFr} (cf. \cite{Ali}-\cite{BCF2},\cite{CXY}-\cite{ChenFeldmanXiang},\cite{ChenZhangZhu}-\cite{ChenYuan},\cite{ElLiu}-\cite{FX},\cite{HKWX,LiZheng,LWY2},\cite{Majda}-\cite{QX},\cite{WY}-\cite{Zhang}).

Multidimensional Riemann problem of the compressible Euler equations, which plays a prominent role in the theory of conservation laws, is one of the core and challenge problems in the mathemaical theory of conservation laws. One important problem is the generalized Riemann problem, which studies the Cauchy problem with discontinuous initial data along a smooth curve. If the data is assumed to be smooth up to the curve, then we expect the solution is of the fan-shape structure. The generalized Riemann problem can also be regarded as the stability of the Riemann solutions of the Cauchy problem with two constant states separeted by a hyperplane. There are a lot of literatures on the local existence of the generalized Riemann problem, for examples, Majda \cite{Majda,Majda1} for the strong shock, Metivier \cite{Me} for the weak shock, Alinhac \cite{Ali} for the rarefaction wave, Coulombel-Secchi \cite{CoSe1} for the two-dimensional vortex sheet, and \cite{chenli} for the two diemsnioal composited waves which can be shocks, rarefaction waves and vortex sheet.

So there is a natural question next:
\emph{how about the global existence of solutions of the generalised Riemann problem}?
As far as we know, there are few results on the global existence of those waves except the ones for the unsteady potential flow equation in $n$-dimensional spaces ($n\ge 5$, see \cite{God}) or in special space-time domains for the potential flow.
So it is of great significance to study the global behaviours of the solutions of the generalised Riemann problem from both the mathematical and physical views.
In this paper, we will give a negative answer, \emph{i.e.}, we will show the solutions of the generalised Riemann problem (if exist locally and is a perturbation of plane shocks) cannot exist globally for the two and three dimensional case, if the flow is isothermal. It means the plane Riemann solutions are not stable globally with respect to a smooth perturbation. Based on it, in order to obtain the global stability, we should think about the generalised Riemann problem in a weak sense. Moreover, the liftspan estimate, which is consistent with the liftspan estimate for the nonlinear wave equations, is also obtained. The result is different from the one dimensional case, in which the global existence is established (cf. \cite{LiZhao,LiWang}).

\section{Generalised Riemann Problem and Main Result}
%
%
%

The multidimensional inviscid compressible flow is governed by the following Euler equation:

\begin{equation}\label{EulerPoly}
\left \{
\begin{aligned}
&\rho_t+\rm{div}(\rho \mathbf{u})=0,\\
&(\rho u)_t+\rm{div}(\rho \mathbf{u}\otimes\mathbf{u})+\nabla p=0,
\end{aligned}\right.
\end{equation}
where $\rho$, $p$ and $\mathbf{u}$ are density, pressure and velocity repectively. For the isothermal flow, the pressure and density satisfy the thermodynamic relation that $p=\rho$. 
In this paper, we are concerned with the global stability/instability of solutions of the generalised Riemann problem governed by equations \eqref{EulerPoly} for the isothermal flow.
Since till now the local existence result for the vortex sheet is only available for the two-dimensional case, we will consider the two-dimensional case first. The global non-stability for the three-dimensional case to the isothermal flow will be proved in the end even though we do not know whether the local nonlinear existence can be obtained or not.

For the two dimensional case, equations \eqref{EulerPoly} become
\begin{equation}\label{problem}
\left \{
\begin{aligned}
&\rho_t+(\rho u)_x+(\rho v)_y=0,\\
&(\rho u)_t+(\rho u^2)_x+(\rho uv)_y+\rho_x=0,\\
&(\rho v)_t+(\rho uv)_x+(\rho v^2)_y+\rho_y=0,
\end{aligned}\right.
\end{equation}
where $(u,v)$ is the velocity
with the initial data
\begin{equation}\label{data}
t=0:\left \{
\begin{aligned}
&\rho=\left \{
\begin{aligned}
&\rho_r+\varepsilon\rho_0(x,y),\qquad x>\varepsilon\Pi(y),\\
&\rho_l+\varepsilon\rho_0(x,y),\qquad x<\varepsilon\Pi(y),\\
\end{aligned} \right.\\
&u=\left \{
\begin{aligned}
&u_r+\varepsilon u_0(x,y),\qquad x>\varepsilon\Pi(y),\\
&u_l+\varepsilon u_0(x,y),\qquad x<\varepsilon\Pi(y),\\
\end{aligned} \right.\\
&v=\varepsilon v_0(x,y),\\
\end{aligned} \right.
\end{equation}
where $\rho_r, \rho_l, u_r, u_l$ are constants, and functions $\Pi(y)\in C_0^{\infty}(\R)$, $\rho_0, u_0, v_0 \in C_0^{\infty}(\R^2)$ satisfy
\begin{equation}\label{2.4x}
supp~\Pi(y)\subset \{y\big||y|\le 1\}
\end{equation}
and
\begin{equation}\label{2.5x}
 supp~ \rho_0,u_0,v_0\subset B_1=\{(x, y): x^2+y^2\leq1\}.
\end{equation}
Moreover, $\varepsilon>0$ is a small parameter.

For a piecewise $C^1$ weak solution of \eqref{problem} with $C^1$-discontinuities, by the integration by part, it is easy to know that the solution is a solution of \eqref{problem} in the classic sense in each smooth subregion, and across the discontinuities ($\Pi$ for example), satisfies the following Rankine-Hugoniot conditions:
\begin{align}
\partial_t\Pi[\rho]\Big|_{\Pi^-}^{\Pi^+}-[\rho u]\Big|_{\Pi^-}^{\Pi^+}+\partial_y\Pi[\rho v]\Big|_{\Pi^-}^{\Pi^+}&=0\label{RH1}\\
\partial_t\Pi[\rho u]\Big|_{\Pi^-}^{\Pi^+}-[\rho u^2+\rho]\Big|_{\Pi^-}^{\Pi^+}+\partial_y\Pi[\rho uv]\Big|_{\Pi^-}^{\Pi^+}&=0\\
\partial_t\Pi[\rho v]\Big|_{\Pi^-}^{\Pi^+}-[\rho uv]\Big|_{\Pi^-}^{\Pi^+}+\partial_y\Pi[\rho v^2+\rho]\Big|_{\Pi^-}^{\Pi^+}&=0\label{RH3}
\end{align}
where $[\cdot]\Big|_{\Pi^-}^{\Pi^+}$ denotes the difference of the left hand side limit and the right hand side limit of the quantity concerned on the discontinuity $x=\Pi(t,y)$.

For the corresponding one-dimensional Riemann problem, which is governed by the following equaiton
 \begin{equation}\label{1D}
\left \{
\begin{aligned}
&\rho_t+(\rho u)_x=0,\\
&(\rho u)_t+(\rho u^2)_x+\rho_x=0,\\
\end{aligned} \right.
\end{equation}
with the initial data that
\begin{equation}\label{1Ddata}
t=0:\left \{
\begin{aligned}
&\rho=\left \{
\begin{aligned}
&\rho_r,x>0,\\
&\rho_l,x<0,\\
\end{aligned} \right.\\
&u=\left \{
\begin{aligned}
&u_r,x>0,\\
&u_l,x<0,\\
\end{aligned} \right.
\end{aligned} \right.
\end{equation}
it is well-known that if constant vector $(\rho_r,u_r)$ lies in a cornered domain with boundaries being the wave curves starting from $(\rho_l,u_l)$, Riemann problem \eqref{1D} and \eqref{1Ddata}
admits a Riemann solution who consists of three constant states $(\rho_l,u_l)$, $(\rho_m,u_m)$, and $(\rho_r,u_r)$ separated by two shocks with shock speeds $\sigma_{+}$ and $\sigma_{-}$, respectively. Without loss of the generality, we assume $u_m=0$,
otherwise we can introduce the coordinate transformation that $x\rightarrow x-u_mt$.
%
In this case, across the shock, the following Rankine-Hugoniot conditions hold
\begin{align}
\sigma_+(\rho_r-\rho_m)-\rho_ru_r=0,\qquad & \sigma_{+}\rho_ru_r-\rho_ru_r^2-\rho_r+\rho_m=0\label{2.8}\\
\sigma_-(\rho_l-\rho_m)-\rho_lu_l=0,\qquad &\sigma_{-}\rho_lu_l-\rho_lu_l^2-\rho_l
+\rho_m=0.\label{2.9}
\end{align}
Moreover, the Riemann solution satisfies the following entropy condition:
\begin{equation}\label{entropy}
\rho_m>\rho_l,\qquad \rho_m>\rho_r,\qquad u_l>0,\qquad u_r<0.
\end{equation}
In summary, the Riemann solution of equation \eqref{1D} with initial data \eqref{1Ddata} is
\begin{equation}\label{2.11}
\rho=\begin{cases}
\rho_r,\quad&\mbox{if }x>\sigma_{+}t\\
\rho_m,\quad&\mbox{if }\sigma_{-}t<x<\sigma_{+}t\\
\rho_l,\quad&\mbox{if }x<\sigma_t
\end{cases},
\qquad
u=\begin{cases}
u_r,\quad&\mbox{if }x>\sigma_{+}t\\
0,\quad&\mbox{if }\sigma_{-}t<x<\sigma_{+}t\\
u_l,\quad&\mbox{if }x<\sigma_t
\end{cases}.
\end{equation}

The generalized Riemann problem \eqref{problem} with initial data \eqref{data} can be regarded as a small perturbation of the Riemann solution \eqref{2.11} when $\varepsilon$ is sufficiently small.

Under a condition for the speed of the initial data (similar condition as the one in \cite{CoSe1} for the vortex sheet), Chen and Li \cite{chenli} established the local existence of piecewise smooth solution  of equations \eqref{problem} with initial data \eqref{data} containing all three waves (\emph{i.e.}, shock wave, rarefaction wave and contact discontinuity). Their existence result also includes the case that the generalised Riemann solution consists of the 1-shock wave $x=\Pi_-(t,y)$ corresponding to the first eigenvalue, the 3-shock wave $x=\Pi_+(t,y)$ correponding to the third eigenvalue, and the contact discontinuity $x=\Pi_0(t,y)$, with the condition that
\begin{equation}
\varepsilon\Pi(y)=\Pi_-(0,y)=\Pi_0(0,y)=\Pi_+(0,y),
\end{equation}
and that
\begin{equation}\label{novacumm}
0<\rho_\ast<\rho<\rho^\ast<\infty,
\end{equation}
where $\rho_\ast$ and $\rho^\ast$ are two constants.

The aim of this paper is to prove that, for the isothermal case, such piecewise smooth solution  of equations \eqref{problem} with initial data \eqref{data} obtained in \cite{chenli} can not be global in time in general. This result is different from the one for the one dimensional case, in which the global existence is established (cf. \cite{LiZhao,LiWang}). We also obtain the lifespan estimate.

\begin{thm}\label{result} For the given initial data \eqref{data}, assume that
	\begin{equation}\label{2.17}
	\int_{-\infty}^{+\infty}\int_{-\infty}^{+\infty}(e^y+e^{-y})\rho_0dxdy+(\rho_l-\rho_r)\int_{-\infty}^{+\infty}(e^y+e^{-y})\Pi dy\geq0,
	\end{equation}
	and there exists a constant $C>0$ such that
	\begin{equation}\label{2.18}
\begin{aligned}
		&\int_{-\infty}^{+\infty}\int_{\varepsilon\Pi(y)}^{+\infty}(e^y-e^{-y})(\rho_r+\varepsilon\rho_0) v_0dxdy\\
&+	\int_{-\infty}^{+\infty}\int^{\varepsilon\Pi(y)}_{-\infty}(e^y-e^{-y})(\rho_l+\varepsilon\rho_0) v_0dxdy\\
\geq& C.
\end{aligned}
	\end{equation}
If there exist positive constants $R$ and $t_0$ such that solution $(\rho,u,v)$ of equations \eqref{problem} satisfies
\[
\left|(\rho, u, v)-(\rho, u, v)_{l, m, r}\right|_{L^{\infty}}\le C\e, \qquad\mbox{when }x^2+y^2\geq R\mbox{ and }t\geq t_0,
\]
where the constant $C$ in the above two inequalities is positive and does not depend on $\varepsilon$,
then the piecewise smooth solution, whose discontinuities consist of two shock waves $x=\Pi_{\pm}(t,y)$ and a contact discontinuity $x=\Pi_0(t,y)$, for the generalized Riemann problem \eqref{problem} of isothermal compressible Euler equations with initial data \eqref{data} will blow up
in a finite time. Moreover, there exists a positive constant $C$ independent of $\e$ such that the upper bound of the lifespan satisfies the estimate
\begin{equation}\label{lifespan}
T(\e)\leq C\e^{-2}.\\
\end{equation}
\end{thm}

\begin{rem}
The lifespan estimate \eqref{lifespan} is consistent with the lifespan estimate of smooth solutions of nonlinear wave equations in two space dimensions.
\end{rem}

Next, let us consider the three dimensional case.
We denote the coordinates as $(x, y_1, y_2)$,
so the three dimensional compressible isothermal Euler system is
\begin{equation}\label{3dproblem}
\left \{
\begin{aligned}
&(\rho)_t+(\rho u)_x+(\rho v_1)_{y_1}+(\rho v_2)_{y_2}=0,\\
&(\rho u)_t+\left(\rho u^2\right)_x+(\rho uv_1)_{y_1}+(\rho uv_2)_{y_2}+\rho_x=0,\\
&(\rho v_1)_t+\left(\rho uv_1\right)_x+\left(\rho v_1^2\right)_{y_1}+(\rho v_1v_2)_{y_2}+\rho_{y_1}=0,\\
&(\rho v_2)_t+\left(\rho uv_2\right)_x+\left(\rho v_1v_2\right)_{y_1}+\left(\rho v_2^2\right)_{y_2}+\rho_{y_2}=0,
\end{aligned}\right.
\end{equation}
where $(u,v_1,v_2)$ are the velocity. In order to make the notations be consistent with the ones for the two dimensional case, let $v:=(v_1,v_2)$.

\begin{thm}\label{thm3D}
Assume the initial datum satisfy that
\begin{equation}\label{2.21}
\begin{aligned}
&\int_{|\omega|=1}
\int_{\R^2}
e^{y_1\omega_1+y_2\omega_2}
\int_{-\infty}^{+\infty}\rho_0(x,y)dxdyd\sigma\\
&+(\rho_l-\rho_r)\int_{|\omega|=1}
\int_{\R^2}e^{y_1\omega_1+y_2\omega_2}\Pi(y)dyd\sigma\\
\geq&0
\end{aligned}
\end{equation}
and there exists a constant $C>0$ such that
\begin{equation}\label{2.22}
\begin{aligned}
&\int_{|\omega|=1}\int_{\R^2}\int_{\varepsilon\Pi(y)}^{+\infty}e^{y_1\omega_1+y_2\omega_2}(\rho_r+\varepsilon\rho_0) v_0(y)\cdot\omega dxdy\\
&+	\int_{|\omega|=1}\int_{\R^2}\int^{\varepsilon\Pi(y)}_{-\infty}e^{y_1\omega_1+y_2\omega_2}(\rho_l+\varepsilon\rho_0) v_0(y)\cdot\omega dxdy\\
\geq& C.
\end{aligned}
\end{equation}
Then it is impossible that there exists a global piecewise smooth solution of the generalized Riemann problem for the compressible isothermal Euler system \eqref{3dproblem} with initial data \eqref{data} and consisting of two shocks and one contact discontinuity such that
\[
\left|(\rho, u, v)-(\rho, u, v)_{l, m, r}\right|_{L^{\infty}}\le C\e,
\]
where the constant $C$ in the above two inequalities is positive and does not depend on $\varepsilon$.
Furthermore, we have the following lifespan estimate
\begin{equation}\label{3dlifespan}
T(\e)\le \exp\left(C\e^{-1}\right).\\
\end{equation}
\end{thm}

\begin{rem}
	Although there is no result on the local existence of solutions of the generalised Riemann problem due to the nonlinear vortex sheet in three dimensions, we can show the three dimensional solutions of the generalised Riemann problem of isothermal compressible Euler equations cannot exist globally even if one could show the local existence. 
\end{rem}

\begin{rem}
The lifespan estimate \eqref{3dlifespan} is consistent with the lifespan estimate of smooth solutions of nonlinear wave equations in three space dimensions.
\end{rem}

We will show Theorem \ref{result} in Section \ref{sec:3}, and Theorem \ref{thm3D} in Section \ref{sec:5}.


\section{Proof of Theorem \ref{result}: Two dimensional Case}\label{sec:3}
To show Theorem \ref{result}, we will rewrite the first and
third equations in \eqref{problem} in four subdomains separated by the shocks and contact discontinuity, by substracting the background solution. Next, we introduce the multiplier $e^y+e^{-y}$ for the first equation and $e^y-e^{-y}$ for the third equation. Then we can derive an ordinary differential system for two quantities, which are integrals of the solutions with respect to the space variables, by using the Rankine-Hugoniot conditions \eqref{RH1}--\eqref{RH3}. By the delicate analysis of the obtained ordinary differential system, we obtain a blow-up result for the new quantity. Finally, the desired lifespan estimate will be established too.

First, let us introduce a technical lemma. Let
\begin{equation}\label{3.1}
\begin{aligned}
X(t)=&\int_{-\infty}^{+\infty}(e^y+e^{-y})\Bigg[\int_{\Pi_+(t, y)}^{+\infty}(\rho-\rho_{r})dx+\int_{\Pi_0(t, y)}^{\Pi_+(t, y)}(\rho-\rho_{m})dx\\
&+\int_{\Pi_-(t, y)}^{\Pi_0(t, y)}(\rho-\rho_{m})dx+\int_{-\infty}^{\Pi_-(t, y)}(\rho-\rho_{l})dx\Bigg]dy\\
&+(\rho_m-\rho_r)\int_{-\infty}^{+\infty}(e^y+e^{-y})\left[\Pi_+(t, y)-\sigma_+t\right]dy\\
&+(\rho_m-\rho_l)\int_{-\infty}^{+\infty}(e^y+e^{-y})\left[\sigma_-t-\Pi_-(t, y)\right]dy£¬\\
\end{aligned}
\end{equation}
and let
\begin{equation}\label{3.2}
Y(t)=\int_{\R^2}(e^y-e^{-y})\rho vdxdy,
\end{equation}
then we have
\begin{lem}\label{lem:3.1} For the solutions of equations \eqref{problem}, we have the following identities:
\begin{equation}
\label{X}
\begin{aligned}
X'(t)=Y(t)
\end{aligned}
\end{equation}
and
\begin{equation}
\label{Y}
\begin{aligned}
Y'(t)=X(t)+\int_{\R^2}(e^y+e^{-y})\rho v^2dxdy
\end{aligned}
\end{equation}
\end{lem}

\begin{proof}
We rewrite the first equation in \eqref{problem} as
\begin{equation}\label{firsteq}
\left \{\begin{aligned}
&(\rho-\rho_r)_t+(\rho u-\rho_ru_r)_x+(\rho v)_y=0,~~~\mbox{when }\Pi_+(t,y)<x<\infty,\\
&(\rho-\rho_m)_t+(\rho u)_x+(\rho v)_y=0,
     ~~~\mbox{when }  \Pi_-(t,y)<x<\Pi_0(t,y),\\
     &\mbox{ or }  \Pi_0(t,y)<x<\Pi_+(t,y),\\
&(\rho-\rho_l)_t+(\rho u-\rho_lu_l)_x+(\rho v)_y=0,~~~\mbox{when }-\infty<x<\Pi_-(t,y),
\end{aligned}\right.
\end{equation}
where $x=\Pi_+(t, y)$ and $x=\Pi_-(t, y)$ are the right and left shocks respectively, and $x=\Pi_0(t, y)$ is the contact discontinuity.

By the first equation in \eqref{firsteq} and by the integration by parts, we have
\begin{equation}
\label{right}
\begin{aligned}
&\frac{d}{dt}\int_{-\infty}^{+\infty}\int_{\Pi_+(t, y)}^{+\infty}(\rho-\rho_{r})(e^y+e^{-y})dxdy\\
=&-\int_{-\infty}^{+\infty}\partial_t\Pi_+(t, y)(\rho-\rho_{r})(e^y+e^{-y})\Big|_{x=\Pi_+^+(t, y)}dy\\
&+\int_{-\infty}^{+\infty}\int_{\Pi_+(t, y)}^{+\infty}(\rho-\rho_{r})_t(e^y+e^{-y})dxdy\\
=&-\int_{-\infty}^{+\infty}\partial_t\Pi_+(t, y)(\rho-\rho_{r})(e^y+e^{-y})\Big|_{x=\Pi_+^+(t, y)}dy\\
&-\int_{-\infty}^{+\infty}\int_{\Pi_+(t, y)}^{+\infty}\left[(\rho u-\rho_{r}u_r)_x(e^y+e^{-y})+(\rho v)_y(e^y+e^{-y})\right]dxdy\\
=&\int_{-\infty}^{+\infty}(e^y+e^{-y})\Big[-\partial_t\Pi_+(\rho-\rho_r)+(\rho u-\rho_ru_r)
-\partial_y\Pi_+(\rho v)\Big]\Big|_{x=\Pi_+^+(t, y)}dy\\
&+\int_{-\infty}^{+\infty}\int_{\Pi_+(t, y)}^{+\infty}(\rho v)(e^y-e^{-y})dxdy,
\end{aligned}
\end{equation}
where $\Pi_+^+$ denotes that the value taken at $x=\Pi_+(t,y)$ is the limit from the right hand. Next, in the region $\Pi_0(t,y)<x<\Pi_+(t,y)$, by the second equation in \eqref{firsteq}, we have
\begin{equation}
\label{midright}
\begin{aligned}
&\frac{d}{dt}\int_{-\infty}^{+\infty}\int_{\Pi_0(t, y)}^{\Pi_+(t, y)}(\rho-\rho_{m})(e^y+e^{-y})dxdy\\
=&\int_{-\infty}^{+\infty}\partial_t\Pi_+(t, y)(\rho-\rho_{m})(e^y+e^{-y})\Big|_{x=\Pi_+^-(t, y)}dy\\
&-\int_{-\infty}^{+\infty}\partial_t\Pi_0(t, y)(\rho-\rho_{m})(e^y+e^{-y})\Big|_{x=\Pi_0^+(t, y)}dy\\
&-\int_{-\infty}^{+\infty}\int_{\Pi_0(t, y)}^{\Pi_+}\left[(\rho u)_x(e^y+e^{-y})+(\rho v)_y(e^y+e^{-y})\right]dxdy\\
=&\int_{-\infty}^{+\infty}(e^y+e^{-y})\left[\partial_t\Pi_+(\rho-\rho_m)-\rho u+\partial_y\Pi_+(\rho v)\right]\Big|_{x=\Pi_+^-}dy\\
&-\int_{-\infty}^{+\infty}(e^y+e^{-y})\left[\partial_t\Pi_0(\rho-\rho_m)-\rho u+\partial_y\Pi_0(\rho v)\right]\Big|_{x=\Pi_0^+}dy\\
&+\int_{-\infty}^{+\infty}\int_{\Pi_0(t, y)}^{\Pi_+(t, y)}(\rho v)(e^y-e^{-y})dxdy,
\end{aligned}
\end{equation}
where $\Pi_+^-$ and $\Pi_0^+$, similarly as above, denote that the values taken are the limit from the left hand side of $\Pi_+$ and from the right hand side of $\Pi_0$, respectively.

For the integration in the regions $\Pi_-(t,y)<x<\Pi_0(t,y)$ and $-\infty<x<\Pi_-(t,y)$, similarly, we have the following results
\begin{equation}
\label{midleft}
\begin{aligned}
&\frac{d}{dt}\int_{-\infty}^{+\infty}\int_{\Pi_-(t, y)}^{\Pi_0(t, y)}(\rho-\rho_{m})(e^y+e^{-y})dxdy\\
=&\int_{-\infty}^{+\infty}(e^y+e^{-y})\left[\partial_t\Pi_0(\rho-\rho_m)-\rho u+\partial_y\Pi_0(\rho v)\right]\Big|_{x=\Pi_0^-}dy\\
&-\int_{-\infty}^{+\infty}(e^y+e^{-y})\left[\partial_t\Pi_-(\rho-\rho_m)-\rho u+\partial_y\Pi_-(\rho v)\right]\Big|_{x=\Pi_-^+}dy\\
&+\int_{-\infty}^{+\infty}\int_{\Pi_-(t, y)}^{\Pi_0(t, y)}(\rho v)(e^y-e^{-y})dxdy,
\end{aligned}
\end{equation}
and
\begin{equation}
\label{left}
\begin{aligned}
&\frac{d}{dt}\int_{-\infty}^{+\infty}\int_{-\infty}^{\Pi_-(t, y)}(\rho-\rho_{l})(e^y+e^{-y})dxdy\\
=&\int_{-\infty}^{+\infty}(e^y+e^{-y})\left[\partial_t\Pi_-(\rho-\rho_l)-(\rho u-\rho_ru_r)+\partial_y\Pi_-(\rho v)\right]\Big|_{x=\Pi_-^-}dy\\
&+\int_{-\infty}^{+\infty}\int_{-\infty}^{\Pi_-(t, y)}(\rho v)(e^y-e^{-y})dxdy.
\end{aligned}
\end{equation}
We omit the details for the shortness, since the arguments for the two results above are similar to the ones for \eqref{right} and \eqref{midright}.
It follows by adding \eqref{right}-\eqref{left} together that
\begin{equation}
\label{fourparts}
\begin{aligned}
&\frac{d}{dt}\int_{-\infty}^{+\infty}\Bigg[\int_{\Pi_+(t, y)}^{+\infty}(\rho-\rho_{r})dx+\int_{\Pi_0(t, y)}^{\Pi_+(t, y)}(\rho-\rho_{m})dx\\
&+\int_{\Pi_-(t, y)}^{\Pi_0(t, y)}(\rho-\rho_{m})dx+\int_{-\infty}^{\Pi_-(t, y)}(\rho-\rho_{l})dx\Bigg](e^y+e^{-y})dy\\
=&\int_{\R^2}(\rho v)(e^y-e^{-y})dxdy\\
&-\int_{-\infty}^{+\infty}(e^y+e^{-y})\left\{\partial_t\Pi_+[\rho]\Big|_{\Pi_+^-}^{\Pi_+^+}-[\rho u]\Big|_{\Pi_+^-}^{\Pi_+^+}+\partial_y\Pi_+[\rho v]\Big|_{\Pi_+^-}^{\Pi_+^+}\right\}dy\\
&-\int_{-\infty}^{+\infty}(e^y+e^{-y})\left\{\partial_t\Pi_0[\rho]\Big|_{\Pi_0^-}^{\Pi_0^+}-[\rho u]\Big|_{\Pi_0^-}^{\Pi_0^+}+\partial_y\Pi_0[\rho v]\Big|_{\Pi_0^-}^{\Pi_0^+}\right\}dy\\
&-\int_{-\infty}^{+\infty}(e^y+e^{-y})\left\{\partial_t\Pi_-[\rho]\Big|_{\Pi_-^-}^{\Pi_-^+}-[\rho u]\Big|_{\Pi_-^-}^{\Pi_-^+}+\partial_y\Pi_-[\rho v]\Big|_{\Pi_-^-}^{\Pi_-^+}\right\}dy\\
&+\int_{-\infty}^{+\infty}(e^y+e^{-y})\left[\partial_t\Pi_+(\rho_r-\rho_m)-\rho_ru_r\right]dy\\
&+\int_{-\infty}^{+\infty}(e^y+e^{-y})\left[\partial_t\Pi_-(\rho_m-\rho_l)+\rho_lu_l\right]dy,\\
\end{aligned}
\end{equation}
where $[f]\Big|_{\Pi_A^-}^{\Pi_A^+}$ denotes the jump difference between the left hand side limit and the right hand side limit of the function $f$ on the curve $x=\Pi_A(t, y)$
for $A\in \{+, 0, -\}$. By Rankine-Hugoniot conditions \eqref{RH1} and \eqref{2.8}-\eqref{2.9},
we obtain from \eqref{fourparts} that
\begin{equation}
\label{fourparts1}
\begin{aligned}
&\frac{d}{dt}\int_{-\infty}^{+\infty}\Bigg[\int_{\Pi_+(t, y)}^{+\infty}(\rho-\rho_{r})dx+\int_{\Pi_0(t, y)}^{\Pi_+(t, y)}(\rho-\rho_{m})dx\\
&+\int_{\Pi_-(t, y)}^{\Pi_0(t, y)}(\rho-\rho_{m})dx+\int_{-\infty}^{\Pi_-(t, y)}(\rho-\rho_{l})dx\Bigg](e^y+e^{-y})dy\\
=&\int_{\R^2}(\rho v)(e^y-e^{-y})dxdy\\
&-(\rho_m-\rho_{r})\int_{-\infty}^{+\infty}(e^y+e^{-y})\left(\partial_t\Pi_+-\sigma_+\right)dy\\
&-(\rho_m-\rho_{l})\int_{-\infty}^{+\infty}(e^y+e^{-y})\left(\sigma_--\partial_t\Pi_-\right)dy.\\
\end{aligned}
\end{equation}

It is \eqref{X}, based on the observation that
\[
\int_{-\infty}^{+\infty}(e^y+e^{-y})\left(\partial_t\Pi_+-\sigma_+\right)dy=\frac{d}{dt}(\int_{-\infty}^{+\infty}(e^y+e^{-y})\left(\Pi_+-\sigma_+t\right)dy)
\]
and
\[
\int_{-\infty}^{+\infty}(e^y+e^{-y})\left(\sigma_--\partial_t\Pi_-\right)dy=\frac{d}{dt}(\int_{-\infty}^{+\infty}(e^y+e^{-y})\left(\sigma_-t-\Pi_-\right)dy).
\]

\medskip
Now we are going to show \eqref{Y}. As in \eqref{firsteq}, we rewrite the third equation in \eqref{problem} as follows
\begin{equation}\label{thirdeq}
\left \{\begin{aligned}
&(\rho v)_t+(\rho uv)_x+(\rho v^2)_y+(\rho-\rho_r)_y=0,~~\mbox{when }\Pi_+(t,y)<x<\infty,\\
&(\rho v)_t+(\rho uv)_x+(\rho v^2)_y+(\rho-\rho_m)_y=0,
~~\mbox{when }   \Pi_-(t,y)<x<\Pi_0(t,y),\\
&\mbox{ or }  \Pi_0(t,y)<x<\Pi_+(t,y),\\
&(\rho v)_t+(\rho uv)_x+(\rho v^2)_y+(\rho-\rho_l)_y=0,~~\mbox{when }-\infty<x<\Pi_-(t,y).
\end{aligned}\right.
\end{equation}
As above, we divide the following quantity into four parts such that
\[
\begin{aligned}
&\frac{d}{dt}\int_{-\infty}^{+\infty}\int_{-\infty}^{+\infty}(e^y-e^{-y})\rho vdxdy\\
=&\frac{d}{dt}\int_{-\infty}^{+\infty}(\int_{\Pi_+}^{+\infty}+\int_{\Pi_0}^{\Pi_+}+\int_{\Pi_-}^{\Pi_0}+\int_{-\infty}^{\Pi_-})(e^y-e^{-y})\rho vdxdy.
\end{aligned}
\]
The first integration is on the region $\Pi_+(t,y)<x<\infty$. By the integration by parts, we have
\begin{equation}
\label{right1}
\begin{aligned}
&\frac{d}{dt}\int_{-\infty}^{+\infty}\int_{\Pi_+(t, y)}^{+\infty}(\rho v)(e^y-e^{-y})dxdy\\
=&-\int_{-\infty}^{+\infty}\partial_t\Pi_+(t, y)(\rho v)(e^y-e^{-y})\Big|_{x=\Pi_+^+(t, y)}dy
+\int_{-\infty}^{+\infty}\int_{\Pi_+(t, y)}^{+\infty}(\rho v)_t(e^y-e^{-y})dxdy\\
=&-\int_{-\infty}^{+\infty}\partial_t\Pi_+(t, y)(\rho v)(e^y-e^{-y})\Big|_{x=\Pi_+^+(t, y)}dy\\
&-\int_{-\infty}^{+\infty}\int_{\Pi_+(t, y)}^{+\infty}\Big[\left(\rho uv(e^y-e^{-y})\right)_x+\left(\rho v^2(e^y-e^{-y})\right)_y
+\left((\rho-\rho_r)(e^y-e^{-y})\right)_y\Big]dxdy\\
&+\int_{-\infty}^{+\infty}\int_{\Pi_+(t, y)}^{+\infty}(\rho v^2+\rho-\rho_r)(e^y+e^{-y})dxdy,\\
=&\int_{-\infty}^{+\infty}(e^y-e^{-y})\Big[-\partial_t\Pi_+(\rho v)+(\rho uv)-\partial_y\Pi_+(\rho v^2+\rho-\rho_r)\Big]\Big|_{x=\Pi_+^+}dy\\
&+\int_{-\infty}^{+\infty}\int_{\Pi_+(t, y)}^{+\infty}(\rho v^2+\rho-\rho_r)(e^y+e^{-y})dxdy.\\
\end{aligned}
\end{equation}

Similarly, for the other three integrations, after straightforward computation, we have
\begin{equation}
\label{midright1}
\begin{aligned}
&\frac{d}{dt}\int_{-\infty}^{+\infty}\int_{\Pi_0(t, y)}^{\Pi_+(t, y)}(\rho v)(e^y-e^{-y})dxdy\\
=&\int_{-\infty}^{+\infty}(e^y-e^{-y})\Big[\partial_t\Pi_+(\rho v)-(\rho uv)+\partial_y\Pi_+(\rho v^2+\rho-\rho_m)\Big]\Big|_{x=\Pi_+^-}dy\\
&-\int_{-\infty}^{+\infty}(e^y-e^{-y})\Big[\partial_t\Pi_0(\rho v)-(\rho uv)+\partial_y\Pi_0(\rho v^2+\rho-\rho_m)\Big]\Big|_{x=\Pi_0^+}dy\\
&+\int_{-\infty}^{+\infty}\int_{\Pi_0(t, y)}^{\Pi_+(t, y)}(\rho v^2+\rho-\rho_m)(e^y+e^{-y})dxdy,\\
\end{aligned}
\end{equation}
\begin{equation}
\label{midleft1}
\begin{aligned}
&\frac{d}{dt}\int_{-\infty}^{+\infty}\int_{\Pi_-(t, y)}^{\Pi_0(t, y)}(\rho v)(e^y-e^{-y})dxdy\\
=&\int_{-\infty}^{+\infty}(e^y-e^{-y})\Big[\partial_t\Pi_0(\rho v)-(\rho uv)+\partial_y\Pi_0(\rho v^2+\rho-\rho_m)\Big]\Big|_{x=\Pi_0^-}dy\\
&-\int_{-\infty}^{+\infty}(e^y-e^{-y})\Big[\partial_t\Pi_-(\rho v)-(\rho uv)+\partial_y\Pi_-(\rho v^2+\rho-\rho_m)\Big]\Big|_{x=\Pi_-^+}dy\\
&+\int_{-\infty}^{+\infty}\int_{\Pi_-(t, y)}^{\Pi_0(t, y)}(\rho v^2+\rho-\rho_m)(e^y+e^{-y})dxdy,\\
\end{aligned}
\end{equation}
and
\begin{equation}
\label{left1}
\begin{aligned}
&\frac{d}{dt}\int_{-\infty}^{+\infty}\int_{-\infty}^{\Pi_-(t, y)}(\rho v)(e^y-e^{-y})dxdy\\
=&\int_{-\infty}^{+\infty}(e^y-e^{-y})\Big[\partial_t\Pi_-(\rho v)-(\rho uv)+\partial_y\Pi_-(\rho v^2+\rho-\rho_l)\Big]\Big|_{x=\Pi_-^-}dy\\
&+\int_{-\infty}^{+\infty}\int_{-\infty}^{\Pi_-(t, y)}(\rho v^2+\rho-\rho_l)(e^y+e^{-y})dxdy.\\
\end{aligned}
\end{equation}

Therefore, it follows by adding \eqref{right1}-\eqref{left1} together that
\begin{equation}
\label{fourparts2}
\begin{aligned}
&\frac{d}{dt}\int_{\R^2}(e^y-e^{-y})\rho vdxdy\\
=&\int_{-\infty}^{+\infty}\Bigg[\int_{\Pi_+(t, y)}^{+\infty}(\rho-\rho_{r})dx+\int_{\Pi_0(t, y)}^{\Pi_+(t, y)}(\rho-\rho_{m})dx\\
&+\int_{\Pi_-(t, y)}^{\Pi_0(t, y)}(\rho-\rho_{m})dx+\int_{-\infty}^{\Pi_-(t, y)}(\rho-\rho_{l})dx\Bigg](e^y+e^{-y})dy\\
&+\int_{\R^2}(e^y+e^{-y})\rho v^2dxdy\\
&-\int_{-\infty}^{+\infty}(e^y-e^{-y})\left\{\partial_t\Pi_+[\rho v]\Big|_{\Pi_+^-}^{\Pi_+^+}-[\rho uv]\Big|_{\Pi_+^-}^{\Pi_+^+}+\partial_y\Pi_+[\rho v^2+\rho]\Big|_{\Pi_+^-}^{\Pi_+^+}\right\}dy\\
&-\int_{-\infty}^{+\infty}(e^y-e^{-y})\left\{\partial_t\Pi_0[\rho v]\Big|_{\Pi_0^-}^{\Pi_0^+}-[\rho uv]\Big|_{\Pi_0^-}^{\Pi_0^+}+\partial_y\Pi_0[\rho v^2+\rho]\Big|_{\Pi_0^-}^{\Pi_0^+}\right\}dy\\
&-\int_{-\infty}^{+\infty}(e^y-e^{-y})\left\{\partial_t\Pi_-[\rho v]\Big|_{\Pi_-^-}^{\Pi_-^+}-[\rho uv]\Big|_{\Pi_-^-}^{\Pi_-^+}+\partial_y\Pi_-[\rho v^2+\rho]\Big|_{\Pi_-^-}^{\Pi_-^+}\right\}dy\\
&-(\rho_m-\rho_r)\int_{-\infty}^{+\infty}\partial_y\Pi_+(t, y)(e^y-e^{-y})dy\\
&+(\rho_m-\rho_l)\int_{-\infty}^{+\infty}\partial_y\Pi_-(t, y)(e^y-e^{-y})dy.\\
\end{aligned}
\end{equation}

For the last two integrals in \eqref{fourparts2}, by the integration by part, we have
\begin{equation}
\label{ltt1}
\begin{aligned}
&(\rho_m-\rho_r)\int_{-\infty}^{+\infty}\partial_y\Pi_+(t, y)(e^y-e^{-y})dy\\
=&(\rho_m-\rho_r)\int_{-\infty}^{+\infty}\partial_y\left[\Pi_+(t, y)-\sigma_+t\right](e^y-e^{-y})dy\\
=&-(\rho_m-\rho_r)\int_{-\infty}^{+\infty}\left[\Pi_+(t, y)-\sigma_+t\right](e^y+e^{-y})dy,\\
&(\rho_m-\rho_l)\int_{-\infty}^{+\infty}\partial_y\Pi_+(t, y)(e^y-e^{-y})dy\\
=&(\rho_m-\rho_l)\int_{-\infty}^{+\infty}\partial_y\left[\Pi_-(t, y)-\sigma_-t\right](e^y-e^{-y})dy\\
=&-(\rho_m-\rho_l)\int_{-\infty}^{+\infty}\left[\Pi_-(t, y)-\sigma_-t\right](e^y+e^{-y})dy.\\
\end{aligned}
\end{equation}

So by the Rankine-Hugoniot conditions \eqref{RH3} on $\Pi_+$, $\Pi_0$, and $\Pi_-$,  
if follows from \eqref{fourparts2}-\eqref{ltt1} that
\begin{equation}
\label{fourparts3}
\begin{aligned}
&\frac{d}{dt}\int_{\R^2}(e^y-e^{-y})\rho vdxdy\\
=&\int_{\R^2}(e^y+e^{-y})\rho v^2dxdy\\
&+\int_{-\infty}^{+\infty}(e^y+e^{-y})\Bigg[\int_{\Pi_+(t, y)}^{+\infty}(\rho-\rho_{r})dx+\int_{\Pi_0(t, y)}^{\Pi_+(t, y)}(\rho-\rho_{m})dx\\
&+\int_{\Pi_-(t, y)}^{\Pi_0(t, y)}(\rho-\rho_{m})dx+\int_{-\infty}^{\Pi_-(t, y)}(\rho-\rho_{l})dx\Bigg]dy\\
&+(\rho_m-\rho_r)\int_{-\infty}^{+\infty}(e^y+e^{-y})\left[\Pi_+(t, y)-\sigma_+t\right]dy\\
&+(\rho_m-\rho_l)\int_{-\infty}^{+\infty}(e^y+e^{-y})\left[\sigma_-t-\Pi_-(t, y)\right]dy.\\
\end{aligned}
\end{equation}

Therefore, from \eqref{fourparts1} and \eqref{fourparts3} we obtain \eqref{X} and \eqref{Y}.
\end{proof}

Based on Lemma \ref{lem:3.1}, now we can show Theorem \ref{result}.
\begin{proof}[Proof of Theorem \ref{result}]
Based on the entropy condition \eqref{entropy} and the Rankine-Hugoniot conditions \eqref{2.9}, we know that
\begin{equation}
1+u_r<\sigma_+<1+u_m=1\qquad\mbox{and}\qquad-1=-1+u_m<\sigma_{-}<-1+u_l,
\end{equation}
where we have used the assumption  that $u_m=0$.
So if $\varepsilon$ is small, the propagation speed of waves is smaller than $1$.
Due to the assumption of the support of the initial data in \eqref{2.4x} and \eqref{2.5x}, there exists a constant $t_0$ such that
%
the support of the solution $v(t, x, y)$ satisfies
\begin{equation}
\label{suppv1}
\begin{aligned}
supp~v(t, x, y)&\subset \left\{(x, y)\big|x^2+y^2\le (t-t_0+C_0t_0+1)^2\right\}\\
&\triangleq \left\{(x, y)\big|x^2+y^2\le (t+C_1)^2\right\}~~~t\ge t_0,
\end{aligned}
\end{equation}
where constant
\[
C_1=(C_0-1)t_0+1,
\]
is independent of $\e$.

By H\"{o}lder's inequality and \eqref{suppv1}, for $t\geq t_0$, we have
\begin{equation}
\label{Y1}
\begin{aligned}
Y^2(t)\le \int_{x^2+y^2\le (t+C_1)^2}(e^y+e^{-y})\rho dxdy\int_{\R^2}(e^y+e^{-y})\rho v^2dxdy.
\end{aligned}
\end{equation}

Note that
\begin{equation}
\label{Y2}
\begin{aligned}
 &\int_{x^2+y^2\le (t+C_1)^2}(e^y+e^{-y})\rho dxdy\\
 \le &\rho^\ast\int_{|y|\le t+C_1}(e^y+e^{-y})\int_{|x|\le \sqrt{(t+C_1)^2-|y|^2}}dxdy\\
 \le &C(t+C_1)^{\frac12}e^{t}\int_{|y|\le t+C_1}e^{-t+C_1}(e^y+e^{-y})\sqrt{t+C_1-|y|}dy\\
 \le& C(t+1)^{\frac12}e^t,\\
\end{aligned}
\end{equation}
where the last inequality can be obtained by using the variable transformation $\tau=t-|y|$. Here and afterwards, $C$ denotes a generic positive constant which is independent of $\e$. Therefore, it follows from \eqref{X}, \eqref{Y}, \eqref{Y1} and \eqref{Y2} that
\begin{equation}
\label{Y3}
\begin{aligned}
Y'(t)&\ge X(t)+CY^2(t)e^{-t}(t+1)^{-\frac12}\\
&\ge X(0)+\int_0^tY(\tau )d\tau+CY^2(t)e^{-t}(t+1)^{-\frac12},~~~t\ge t_0.\\
\end{aligned}
\end{equation}

Let
\begin{equation}\nonumber
e^tZ=\int_0^tY(\tau )d\tau.
\end{equation}
Similarly as done for \eqref{Y3}, it is easy to get from \eqref{X}, \eqref{Y} and \eqref{Y11} that
\begin{equation}\label{Z3}
Z''+2Z'\ge e^{-t}X(0)+e^{-t}Y^2(t)\left(\int_{|x|^2+|y|^2\leq (C_0t+1)^2}(e^y+e^{-y})dxdy\right)^{-1}.\\
\end{equation}

In fact, by the finite propagation speed of waves, we also know that the support of solution $v(t,x,y)$ satisfies that for all $t\geq0$
\begin{equation}
supp~v(t, x, y)\subset \left\{(x, y)\big|x^2+y^2\le (C_0t+1)^2\right\}.
\end{equation}
So by the H\"{o}lder inequality, we also have for all $t\geq0$
\begin{equation}
\label{Y11}
\begin{aligned}
Y^2(t)\le \int_{x^2+y^2\le (C_0t+1)^2}(e^y+e^{-y})\rho dxdy\int_{\R^2}(e^y+e^{-y})\rho v^2dxdy.
\end{aligned}
\end{equation}
Then \eqref{Y11} follows by exactly the same argument as the one for \eqref{Y3}.

By assumption \eqref{2.17}, we know that $X(0)\ge 0$. Then it follows from \eqref{Z3} that
\begin{equation}\label{Z1}
(e^{2t}Z')'\ge 0,~~~t\ge 0.
\end{equation}

By assumption \eqref{2.18}, we know that $Z'(0)=Y(0)\geq C\varepsilon$. It means that $Z'(0)=Y(0)\ge 0$, so \eqref{Z1} implies for $t\geq0$,
\begin{equation}\nonumber
e^{2t}Z'\ge 0.
\end{equation}

Therefore for $t\geq0$, we have
\begin{equation}\label{Z2}
Z'\ge 0,\qquad\mbox{and}\qquad Z+Z'\ge Z+\frac{Z'}{2}.
\end{equation}

Also, \eqref{Y3} yields that
\begin{equation}\label{Z}
Z''+2Z'\ge e^{-t}X(0)+C(Z+Z')^2(t+1)^{-\frac12},~~~t\ge t_0.\\
\end{equation}
Let $W=Z'+2Z$, we finally get from \eqref{Z} and \eqref{Z2} that for $t\ge t_0$
\begin{equation}\label{W}
\begin{aligned}
W'&\ge CW^2(1+t)^{-\frac12}+e^{-t}X(0)\\
&\ge CW^2(1+t)^{-\frac12}.\\
\end{aligned}
\end{equation}
Noting that for $t_0\ge 1$, we have
\[
\begin{aligned}
W(t_0)&=Z'(t_0)+2Z(t_0)\\
&\ge Z(1)\\
&= e^{-1}\int_0^1Y(t)dt\\
&\ge CY(0).
\end{aligned}
\]

Because from \eqref{2.18}, we know that $Y(0)\ge C\e$.
So by \eqref{W}, we know that
\[
\frac{W'}{W^2}\geq\frac{C}{(1+t)^{\frac{1}{2}}}.
\]
Then
\[
W(t)\geq\frac{1}{\frac{1}{W(t_0)}-2C(1+t)^{\frac{1}{2}}}.
\]

Therefore, $W(t)$ will blow up before a time $C\varepsilon^{-2}$. It is the lifespan estimate \eqref{lifespan}.
\end{proof}


\section{Proof of Theorem \ref{thm3D}: Three dimensional case}\label{sec:5}
In this section, we will prove Theorem \ref{thm3D}. In order to make it, instead of the test function $e^{y}\pm e^{-y}$ used in two dimensions, we introduce the following test function
\[
F(y)=\int_{\omega_1^2+\omega_2^2=1}e^{y_1\omega_1+y_2\omega_2}d\sigma.
\]
Test function $F(y)$ is radially symmetric and satisfies the following properties
\begin{equation}\label{test}
\left \{
\begin{aligned}
&\Delta_y F(y)=F(y),\\
&0\le F(y)\le Cr^{-\frac12}e^r,~~\mbox{where }r=|y|=\sqrt{y_1^2+y_2^2}.\\
\end{aligned}\right.
\end{equation}

One can refer the reference \cite{LZ} for more details of the properties of the test function $F(y)$. Based on the test function $F(y)$, now we can prove Theorem \ref{thm3D}.

\begin{proof}[Proof of Theorem \ref{thm3D}]
Let
\[
\begin{aligned}
X(t)=&
\int_{\R^2}F(y)
\Bigg[\int_{\Pi_+(t, y)}^{+\infty}(\rho-\rho_{r})dx+\int_{\Pi_0(t, y)}^{\Pi_+(t, y)}(\rho-\rho_{m})dx\\
&\qquad\qquad+\int_{\Pi_-(t, y)}^{\Pi_0(t, y)}(\rho-\rho_{m})dx+\int_{-\infty}^{\Pi_-(t, y)}(\rho-\rho_{l})dx\Bigg]
dy\\
&+(\rho_m-\rho_r)
\int_{\R^2}
F(y)\left[\Pi_+(t, y)-\sigma_+t\right]
dy\\
&+(\rho_m-\rho_l)
\int_{\R^2}F(y)
\left[\sigma_-t-\Pi_-(t, y)\right]
dy,
\end{aligned}
\]
and
\[
\begin{aligned}
Y(t)&=\int_{|\omega|=1}\int_{\R^2}\int_{-\infty}^{+\infty}e^{y_1\omega_1+y_2\omega_2}\rho v_\omega dxdyd\sigma,\\
\end{aligned}
\]
where
\[
v_\omega(t, x, y_1, y_2)=\omega_1v_1(t, x, y_1, y_2)+\omega_2v_2(t, x, y_1, y_2).\\
\]
Here $(u,v_1,v_2)$ are the velocity.

Multiplying the third and forth equations in \eqref{3dproblem} with $\omega_1$ and $\omega_2$ respectively, and then adding together,
we come to a new system
\begin{equation}\label{3dproblemnew}
\left \{
\begin{aligned}
&(\rho)_t+(\rho u)_x+(\rho v_1)_{y_1}+(\rho v_2)_{y_2}=0,\\
&(\rho u)_t+\left(\rho u^2\right)_x+(\rho uv_1)_{y_1}+(\rho uv_2)_{y_2}+\rho_x=0,\\
&(\rho v_\omega)_t+\left(\rho uv_\omega\right)_x+\left(\rho v_1v_\omega\right)_{y_1}+(\rho v_2v_\omega)_{y_2}+\omega_1\rho_{y_1}
+\omega_2\rho_{y_2}=0.\\
\end{aligned}\right.
\end{equation}

Similar to the two dimensional case, by a straightforward computation samely as the one for the proof of Lemma \ref{lem:3.1} (we omit the long but tedius details for the shortness),  we can establish the following ordinary differential system
\begin{equation}\label{3dode}
\begin{aligned}
&X'(t)=Y(t),\\
&Y'(t)=X(t)+\int_{|\omega|=1}\int_{\R^2}\int_{-\infty}^{+\infty}e^{y_1\omega_1+y_2\omega_2}\rho v_\omega^2 dxdyd\sigma.\\
\end{aligned}
\end{equation}

Let
\[
e^{t}Z(t)=\int_0^tY(\tau)d\tau.\\
\]
Then by the H\"{o}lder inequality and \eqref{3dode}, we get
\begin{equation}\label{3dZ}
\begin{aligned}
&Z''+2Z'\\
=&e^{-t}X(0)+e^{-t}\int_{|\omega|=1}\int_{\R^2}\int_{-\infty}^{+\infty}e^{y_1\omega_1+y_2\omega_2}\rho v_\omega^2 dxdyd\sigma\\
\ge &e^{-t}X(0)+e^{-t}Y^2(t)\left(\int_{|\omega|=1}\int_{x^2+r^2\le (t+R)^2}e^{y_1\omega_1+y_2\omega_2}\rho dxdyd\sigma\right)^{-1}.\\
\end{aligned}
\end{equation}
For the last inequality above we use the properties that $v_1$ and $v_2$ are compactly supported and the finite propogation speed. More preciesly, by the entropy condition, the shock speeds $\sigma_{+}<1$ and $\sigma_->-1$. When $\Pi_-(t,y)\leq x\leq \Pi_+(t,y)$, the speed of the characteristic at the boundary of the compact support of the solution is $1$ since $(u,v_1,v_2)=(0,0,0)$. Hence if $\varepsilon$ is sufficiently small, we know that there exists $R>0$ large enough which does not depend on the data such that the support of $v_1$ and $v_2$ satisfies that $x^2+r^2\le (t+R)^2$, where $r^2=y^2_1+y_2^2$.

By \eqref{2.21}, we know that $X(0)\geq 0$. So \eqref{3dZ} implies
\[
\left(e^{2t}Z'\right)'\ge 0,~~t\ge 0.
\]

By \eqref{2.22}, we further know that that $Z'(0)=Y(0)\geq 0$. So it follows from the inequality above that
\[
e^{2t}Z'\ge 0,~~t\ge 0.
\]

Therefore
\begin{equation}\label{3dZ2}
\begin{aligned}
&Z'\ge 0,~~~t\ge 0\\
&Z+Z'\ge Z+\frac{Z'}{2},~~~t\ge 0.
\end{aligned}
\end{equation}

Next, we need to estimate of the last term in \eqref{3dZ} when $t\ge \widetilde{t}_0$ for some
fixed time $\widetilde{t}_0$ which is independent of $\e$. But unlike the two diemsnional case, it has to be done in a different way. 
By \eqref{test}, we have
\begin{equation}\label{lastterm}
\begin{aligned}
&\int_{|\omega|=1}\int_{x^2+r^2\le (t+R)^2}e^{y_1\omega_1+y_2\omega_2}\rho dxdyd\sigma\\
=&\int_{x^2+r^2\le (t+R)^2}F\rho dxdy\\
\le&C(t+1)^{\frac12}\int_{r\le \sqrt{(t+R)^2-|x|^2}}\rho Fr\sqrt{t+R-r}dr\\
\le&C(t+1)^{\frac12}e^t\int_{r\le \sqrt{(t+R)^2-|x|^2}}e^{-t-R+r}r^{\frac12}\sqrt{t+R-r}dr\\
\le&C(t+1)e^t\int_{r\le \sqrt{(t+R)^2-|x|^2}}e^{-t-R+r}\sqrt{t+R-r}dr\\
\le&C(t+1)e^t,~~~t\ge \widetilde{t}_0.\\
\end{aligned}
\end{equation}

Finally, let $W=Z'+2Z$, then from \eqref{2.21}, \eqref{3dZ} and \eqref{3dZ2}, we have
\begin{equation}\label{3dW}
\begin{aligned}
W'&\ge CW^2(1+t)^{-1}+e^{-t}X(0)\\
&\ge CW^2(1+t)^{-1},~~~t\ge \widetilde{t}_0.\\
\end{aligned}
\end{equation}

Without loss of the generality, we can assume $\widetilde{t}_0\ge 1$. So by \eqref{3dZ2} we have
\[
\begin{aligned}
W(\widetilde{t}_0)&=Z'(\widetilde{t}_0)+2Z(\widetilde{t}_0)\\
&\ge Z(1)\\
&= e^{-1}\int_0^1Y(t)dt\\
&\ge CY(0).
\end{aligned}
\]

By \eqref{2.22}, we know that there exists a constant $C$ which does not depend on the data such that $Y(0)\ge C\e$. Then it follows \eqref{3dW} that
\[
W(t)\geq\frac{1}{\frac{1}{W(t_0)}-\ln\frac{1+t}{1+t_0}}\geq \frac{1}{\frac{1}{C\varepsilon}-\ln\frac{1+t}{1+t_0}}.
\]

Therefore, $W(t)$ becomes infinite before the time $\exp\left(C\e^{-1}\right)$.
\end{proof}

\section*{Acknowledgment}
\par\quad
Part of this work was finished when the first author visited the institute of mathematical sciences, he wants to express his sincere thank to Prof. Zhouping Xin for his kind invitation and warm hospitality.
N. A. Lai was partially supported by Zhejiang Province
Science Foundation(LY18A010008), NSFC(11501273, 11726612, 11771359,
11771194), Chinese Postdoctoral Science Foundation(2017M620128, 2018T110332), the Scientific Research Foundation of the First-Class Discipline of Zhejiang Province
(B)(201601). W. Xiang was supported in part by the Research
Grants Council of the HKSAR, China (Project No. CityU 21305215, Project No. CityU 11332916, Project No. CityU 11304817 and Project No. CityU 11303518).
Y. Zhou was supported by Key Laboratory of Mathematics for Nonlinear Sciences (Fudan University), Ministry of Education of China, P.R.China.
Shanghai Key Laboratory for Contemporary Applied Mathematics, School of Mathematical Sciences, Fudan University, P.R. China, NSFC (grants No. 11421061, grants No.11726611, grants No. 11726612), 973 program (grant No. 2013CB834100) and 111 project.

\section*{References}
\bibliographystyle{plain}

\begin{thebibliography}{20}

\bibitem{Ali}
S. Alinhac, Existence d'ondes de rarefaction pour des syst\`{e}mes quasi-lin\'{e}aires hyperboliques multidimensionnels. \textit{Comm. Partial Differential Equations}, \textbf{14} (1989), 173--230.

\bibitem{BaeChenFeldman}
M. Bae, G.-Q. Chen and M. Feldman,
Regularity of solutions to regular shock reflection for potential flows. \textit{Invent. Math.}, \textbf{175} (2009), 505--543.

\bibitem{BCF2}
M. Bae, G.-Q. Chen, and M. Feldman,
Prandtl-Meyer reflection for supersonic flow past a solid ramp.
\textit{Quart. Appl. Math.}, \textbf{71} (2013), 583--600.

\bibitem{BiBr}
S. Bianchini and A. Bressan,
Vanishing viscosity solutions of nonlinear hyperbolic systems.
\textit{Ann. of Math.}, \textbf{161} (2005), 223--342.
	
\bibitem{Br}
A. Bressan, Hyperbolic Systems of Conservation Laws. Oxford: Oxford
University Press, 2000.

\bibitem{BrLiuYang}
A. Bressan, T.-P. Liu and T. Yang,
$L^1$ stability estimates for n$\times$n conservation laws.
\textit{Arch. Rational Mech. Anal.}, \textbf{149} (1999), 1--22.

\bibitem{CXY} G. Cao, W. Xiang and X. Yang, Global structure of admissible solutions of multi-dimensional non-homogeneous scalar conservation law with Riemann-type data. \textit{J. Differential Equations}, \textbf{263} (2017), 1055--1078.

\bibitem{ChenChZhang}
G.-Q. Chen, C. Christoforou and Y. Zhang, Continuous dependence of
entropy solutions to the Euler equations on the adiabatic exponent and Mach
number.
\textit{Arch. Ration. Mech. Anal.}, \textbf{189} (2008), 97--130.

\bibitem{cdx-1}
G.-Q.~Chen, X.~Deng and W.~Xiang.
\newblock The global existence and optimal regularity of solutions for shock
diffraction problem to the nonlinear wave systems,
\newblock {\em Arch. Ration. Mech. Anal.}, {\bf 211} (2014), 61--112.

\bibitem{ChenFeldman}
G.-Q. Chen and M. Feldman, Global solutions to shock reflection by large-angle wedges for potential flow. \textit{Ann. of Math.}, \textbf{171} (2010), 1067--1182.

\bibitem{ChenFeldman1}
G.-Q. Chen and M. Feldman, Mathematics of Shock Reflection-Diffraction and von Neumann Conjectures. Princeton University Press, Princeton, 2018.

\bibitem{CFHX}
G.-Q. Chen, M. Feldman, J. Hu and W. Xiang,
Loss of Regularity of Solutions of the Lighthill Problem for Shock Diffraction for Potential Flow,
\textit{arXiv:1705.06837}, 2017.

\bibitem{ChenFeldmanXiang}
G.-Q. Chen, M. Feldman and W. Xiang,
Convexity of Self-Similar Transonic Shocks and Free Boundaries for Potential Flow. \emph{ arXiv:1803.02431}, 2018.

\bibitem{ChenXiangZhang}
G.-Q. Chen, W. Xiang and Y. Zhang, Weakly Nonlinear Geometric Optics for Hyperbolic Systems of Conservation Laws. \textit{Comm. Partial Differential Equations}, \textbf{38} (2013), 1936--1970.

\bibitem{ChenZhangZhu}
G.-Q. Chen, Y. Zhang and D. Zhu, Existence and stability of supersonic Euler flows past Lipschitz wedges, \textit{Arch. Ration. Mech. Anal.}, \textbf{181} (2006), 261--310.

\bibitem{Chens}
S. Chen, Stability of a Mach configuration. \textit{Comm. Pure Appl. Math.}, \textbf{59} (2006), 1--35.

\bibitem{Chens1}
S. Chen, Mach configuration in pseudo-stationary compressible flow. \textit{J. AMS}, \textbf{21} (2008), 63--100.

\bibitem{chenli}
S. Chen and D. Li, Cauchy problem with general discontinuous initial data along a smooth curve for 2-d Euler system. \textit{J. Differential Equations}, \textbf{257} (2014), 1939--1988.

\bibitem{ChenXinYin}
S. Chen, Z. Xin and H. Yin, Global shock waves for the supersonic flow past a perturbed cone. \textit{Comm. Math. Phys.}, \textbf{228} (2002), 47--84.

\bibitem{ChenYuan}
S. Chen and H. Yuan, Transonic shocks in compressible flow passing a duct for three-dimensional Euler systems. \textit{Arch. Rational Mech. Anal.}, \textbf{187} (2008), 523--556.

\bibitem{ChMiao}
D. Christodoulou and S. Miao, Compressible Flow and Euler's Equations. International Press, Boston, 2014.

\bibitem{CoFr}
R. Courant and K. Friedrichs, Supersonic flow and shock waves. Wiley Interscience, New York, 1948.


\bibitem{CoSe1}
J. Coulombel and P. Secchi, Nonlinear compressible vortex sheets in two-dimensions. \textit{Ann. Sci. \'{E}c. Norm Sup\'{e}r}, \textbf{41} (2008), 85--139.

\bibitem{Da}
C. Dafermos, Hyperbolic conservation laws in continuum physics, Third edition, Springer-Verlag, Heidelberg 2010.

\bibitem{ElLiu}
V. Elling and T.-P. Liu, Supersonic flow onto a solid wedge. \textit{Comm. Pure Appl. Math.}, \textbf{61} (2008), 1347--1448.

\bibitem{FRX}
L. Fan, L. Ruan and W. Xiang, Asymptotic stability of a composite wave of two viscous shock waves for the one-dimensional radiative Euler equations.
\textit{Ann. I.H. Poincare-An.}, (2018), accepted.

\bibitem{FangLiuYuan}
B. Fang, L. Liu, and H. Yuan, Global uniqueness of steady transonic shocks in two-dimensional compressible Euler flows.
\textit{Arch. Rational Mech. Anal.}, \textbf{207} (2013), 317--345.

\bibitem{FX}
B. Fang and W. Xiang,
The uniqueness of transonic shocks in supersonic flow past a 2-D wedge.
\emph{J. Math. Anal. Appl.}, {\bf 437} (2016), 194--213.

\bibitem{Glimm}
J. Glimm, Solutions in the large for nonlinear hyperbolic systems of equations. \textit{Comm. Pure Appl. Math.}, \textbf{18} (1965), 697--715.

\bibitem{God}
P. Godin, Long time existence of a class of perturbation of planar shock fronts for second order hyperbolic conservation laws.
\emph{Duke Math. J.}, {\bf 60} (1990), 425--463.

\bibitem{HKWX}
F. Huang, J. Kuang, D. Wang, and W. Xiang,
Stability of supersonic contact discontinuity for two-dimensional steady compressible Euler flows in a finite nozzle.
\emph{arXiv:1804.04769}, 2018.

\bibitem{Lax1}
P. Lax, Hyperbolic systems of conservation laws. \emph{Comm. Pure Appl. Math.}, \textbf{10} (1957), 537--566.

\bibitem{lax}
P. Lax, Development of singularities of solutions of nonlinear hyperbolic partial differential equations. \emph{J. Mathematical Phys.}, \textbf{5} (1964), 611--614.

\bibitem{LiZheng}
J. Li and Y. Zheng, Interaction of rarefaction waves of the two-dimensional self-similar Euler equations. \emph{Arch. Ration. Mech. Anal.}, \textbf{193} (2009), 523--557.


\bibitem{LWY2}
J. Li, I. Witt, and H. Yin, Global multidimensional shock waves for 2-D and 3-D unsteady potential flow equations.
\emph{SIAM J. Math. Anal.}, \textbf{50} (2018), 933--1009.



\bibitem{LiZhao}
T.-T. Li and Y. Zhao, Global shock solutions to a class of piston problems for the system of one-dimensional isentropic flow.
\emph{Chinese Ann. Math. Ser. B}, \textbf{12} (1991), 495--499.

\bibitem{LiZhouKong}
T.-T. Li, Y. Zhou and D. Kong, Weak linear degeneracy and global classical solutions for general quasilinear hyperbolic systems. \emph{Comm. Partial Differential Equations}, \textbf{19} (1994), 1263--1317.

\bibitem{LiWang}
T.-T. Li and L. Wang, Global propagation of regular nonlinear hyperbolic waves. Progress in Nonlinear Differential Equations and their Applications, 76. Birkhäuser Boston, 2009.

\bibitem{LZ}
T.-T. Li and Y. Zhou, Nonlinear Wave Equations(in Chinese), Series in Contemporary Mathematics, Shanghai Scientific \& Technical Publishers, 2016.

\bibitem{Liu1}
T.-P. Liu, Large time behavior of solutions of initial and initial-boundary value problems of a general system of hyperbolic conservation laws. \emph{Comm. Math. Phys.}, \textbf{55} (1977), 163--177.

\bibitem{Liu}
T.-P. Liu, The development of singularities in the nonlinear waves for quasi-linear hyperbolic partial differential equations. \emph{J. Differential Equations}, \textbf{33} (1979), 92--111.

\bibitem{LiuYang}
T.-P. Liu and T. Yang, $L^1$ stability of weak solutions for 2$\times$2 systems of hyperbolic conservation laws. \emph{J. AMS}, \textbf{12} (1999), 729--774.

\bibitem{Majda}
A. Majda, The stability of multi-dimensional shock fronts. \emph{Memoirs AMS}, \textbf{41} (1983), No. 275.

\bibitem{Majda1}
A. Majda, The existence of multi-dimensional shock fronts. \emph{Memoirs AMS}, \textbf{43} (1983), No. 281.

\bibitem{Me}
G. Metivier, Stability of multi-dimensional weak shocks. \emph{Comm. Partial Differential Equations}, \textbf{15} (1990), 983--1028.

\bibitem{QX}
A. Qu and W. Xiang.
{Three-dimensional steady supersonic Euler flow past a concave cornered wedge with lower pressure at the downstream},
{\it Arch. Ration. Mech. Anal.} {\bf 228} (2018), 431--476.

\bibitem{Si}
T. Sideris, Formation of singularities in three-dimensional compressible fluids. \emph{Comm. Math. Phys.}, \textbf{101} (1985), 475--485.

\bibitem{WY}
Y. Wang and F. Yu, Structural stability of supersonic contact discontinuities in three-dimensional compressible steady flows. \emph{SIAM J. Math. Anal.}, \textbf{47} (2015), 1291--1329.

\bibitem{WYuan}
Y. Wang and H. Yuan, Weak stability of transonic contact discontinuities in three-dimensional steady non-isentropic compressible Euler flows. \emph{Z. Angew. Math. Phys.}, \textbf{66} (2015), 341--388.

\bibitem{XZZ}
W. Xiang, Y. Zhang and Q. Zhao, Two-dimensional steady supersonic exothermically reacting Euler flows with strong contact discontinuity over Lipschitz wall. \emph{Interface Free Bound.}, (2018), accepted.


\bibitem{Zhang}
Y. Zhang, Steady supersonic flow past an almost straight wedge with large vertex angle. \emph{J. Differential Equations}, \textbf{192} (2003), 1--46.


\end{thebibliography}

\end{document}